\documentclass{amsart}
\usepackage{amsmath, amsthm, amscd, amsfonts, amssymb, color}

\usepackage{amsfonts}
\usepackage{amsmath,amscd}
\usepackage{amssymb}
\usepackage{amsthm}
\usepackage{newlfont}

\long\def\alert#1{\parindent2em\smallskip\hbox to\hsize%
{\hskip\parindent\vrule%
\vbox{\advance\hsize-2\parindent\hrule\smallskip\parindent.4\parindent%
\narrower\noindent#1\smallskip\hrule}\vrule\hfill}\smallskip\parindent0pt}
 \newtheorem{thm}{Theorem}[section]
 \newtheorem{cor}[thm]{Corollary}
 \newtheorem{lem}[thm]{Lemma}
 \newtheorem{prop}[thm]{Proposition}
 \theoremstyle{definition}
 \newtheorem{defn}[thm]{Definition}
 \theoremstyle{remark}

 \newtheorem{ex}[thm]{Example}
 
 \numberwithin{equation}{section}

\begin{document}

%
\title[ Schur multiplier of nilpotent Lie algebras]
{Some results on the Schur multiplier of nilpotent Lie algebras}
\author[P. Niroomand]{Peyman Niroomand}
\address{School of Mathematics and Computer Science\\
Damghan University, Damghan, Iran}
\email{niroomand@du.ac.ir, p$\_$niroomand@yahoo.com}

\author[F. Johari]{Farangis Johari}
\address{Department of Pure Mathematics\\
Ferdowsi University of Mashhad, Mashhad, Iran}
\email{farangisjohary@yahoo.com}
%
%



\date{\today}

\begin{abstract} For a non-abelian Lie algebra $L$ of dimension $n$ with the derived subalgebra of dimension $m$ , the author earlier proved that the dimension of its Schur multiplier is bounded by  $\frac{1}{2}(n+m-2)(n-m-1)+1$. In the current work, we obtain the class of all nilpotent Lie algebras which attains the above bound. Furthermore, we also improve this bound as much as possible.

 \end{abstract}

\maketitle

\section{Introduction,  Motivation and Preliminaries}
Analogous to the Schur multiplier of a group, the Schur multiplier of a Lie algebra, $\mathcal{M}(L)$, can be defined as $ \mathcal{M}(L)\cong R\cap F^2/[R,F]$ where $L\cong F/R $ and $ F $ is a free Lie algebra (see  \cite{ellis, ni, ni1} for more information).

There are several works to show that the results on the Schur multiplier of finite $p$-groups ($p$ a
prime) have analogues on the Schur multiplier $\mathcal{M}(L)$ of a nilpotent Lie algebra $L$ of
dimension $n$. For instance, in \cite[Theorem 3.1]{ni}, the author proved that for a non-abelian nilpotent Lie algebra of dimension $n$, we have
\begin{equation}\label{eq}
\dim \mathcal{M}(L)\leq \frac{1}{2}(n+m-2)(n-m-1)+1
\end{equation}and the equality holds when $L\cong H(1)\oplus A(n-3),$
 in where $H(m)$ and $A(n)$ denote
the Heisenberg Lie algebra of dimension $ 2m+1$ (a Lie algebra such that $L^2 = Z(L)$ and $\dim L^2 = 1$) and abelian Lie algebra of dimension $n$.
This improves the earlier results obtained by the same author in \cite[Main Theorem]{ni0} for Lie algebras.
Recently, the structure of all $p$-groups of class two for which $|\mathcal{M}(G)|$ attains the bound \cite[Theorem 3.1]{ni} is classified in \cite{rai2}, and then \cite[Theorem 3.1]{ni} has been improved by Hatui in \cite{H}.

In the present article,  we give some new inequalities on the exterior square and the Schur multiplier of Lie algebras. Then we classify all
nilpotent Lie algebras that attain the bound \ref{eq}. More precisely, they are exactly nilpotent Lie algebras of class two. Moreover, for nilpotent Lie algebras of class at last $3,$ we improve \ref{eq} as much as possible. It develops some key results of \cite{H, rai2} for the class of Lie algebras by a different way.

For the convenience of the reader, we give some results without proofs which will be used in the next section. For a Lie algebra $L,$ we use notation $L^{(ab)}$ instead of $L/L^2$.
\begin{lem}\label{1kg}
 \cite[Proposition 3]{ba1}
Let $ A$ and $ B $ be two Lie algebras. Then
\[ \mathcal{M}(A\oplus B)\cong   \mathcal{M}(A) \oplus  \mathcal{M}(B) \oplus (A^{(ab)}\otimes_{mod} B^{(ab)}), \] in where $A^{(ab)}\otimes_{mod} B^{(ab)}$ is the standard tensor product $A$ and $B$.
\end{lem}

Schur multipliers of abelian and Heisenberg Lie algebras are well known. See for instance
  \cite[Lemma 2.6]{ni3}.

\begin{lem}\label{2}
We have
\begin{itemize}
\item[$(i)$]$\dim \mathcal{M}(A(n))=\dfrac{1}{2}n(n-1).$
\item[$(ii)$]$\dim \mathcal{M}(H(1))=2.$
\item[$(iii)$]$\dim \mathcal{M}(H(m))=2m^2-m-1$ for all $ m\geq 2.$
\end{itemize}
\end{lem}
Our next aim is to exhibit a close relation between the Lie algebra $\mathcal{M}(L)$ and $\mathcal{M}(L/K)$ where
$K$ is an ideal of $L$.

\begin{lem}\cite[Corollary 2.3]{ni}\label{d1} Let $L$ be a finite dimensional Lie algebra, $K$ an ideal of $L$
and $H = L/K.$ Then\[
\dim \mathcal{M}(L)+\dim(L^2\cap K) \leq \dim \mathcal{M}(H) +\dim \mathcal{M}(K)+\dim(H^{(ab)} \otimes_{mod} K^{(ab)}).\]
\end{lem}
The next lemma gives an upper bound for the Schur multiplier of an $n$-dimensional nilpotent Lie algebra with the
derived subalgebra of maximum dimension.

\begin{lem}\cite[Theorem 3.1]{ni1}\label{35}
 An $n$-dimensional nilpotent Lie algebra $ L$ in which $\dim L^2 = n - 2$
and $n \geq 4$ has $\dim \mathcal{M}(L) \leq \dim L^2.$
\end{lem}
The following theorem improves the earlier bound on the dimension of Schur multiplier.
\begin{thm}\cite[Theorem 3.1]{ni}\label{5} Let $L$ be an $n$-dimensional non-abelian nilpotent Lie algebra with the derived subalgebra of dimension $m.$ Then \[\dim \mathcal{M}(L)\leq \dfrac{1}{2}(n+m-2)(n-m-1)+1.\]
In particular, when $m = 1$ the bound is attained if and only if $L=H(1)\oplus A(n-3)$.
\end{thm}
From \cite{ellis} $L\wedge L$ and $L\otimes L$ denote the exterior square and the tensor square of a Lie algebra $L$, respectively.
The authors assume that the reader is familiar with these concepts.
\begin{prop}\cite[Proposition 1.1]{ellis}\label{lk}
Let $ L $ be a Lie algebra such that $I  $ and $ K $ are ideals in $ L$ and $ K\subseteq I.$ Then
 the sequence $K\wedge L\rightarrow I\wedge L\rightarrow L/K\wedge I/K\rightarrow 0$ is exact.
\end{prop}
\begin{lem}\cite[Theorem 35 $(iii)$]{ellis}\label{j1}
Let $ L $ be a Lie algebra. Then
$0\rightarrow \mathcal{M}(L)\rightarrow L\wedge L \xrightarrow{\kappa'} L^2\rightarrow 0$ is exact, in where
$\kappa': L\wedge L \rightarrow L^2$ is given by $l\wedge l_1\mapsto [l,l_1].$

\end{lem}
The following lemmas are useful in our main results.
\begin{lem}\cite[Lemma 2.14]{c}\label{943}
Let $I$ be a central ideal of  $L.$ Then
\[ I\wedge L\cong \big(I\otimes_{mod} L/L^2\big) / \langle  x\otimes (x+L^2)| x\in I\rangle.\]
Moreover, if $I\subseteq L^2,$ then $ I\wedge L\cong I\otimes_{mod} L/L^2.$
\end{lem}
The next lemma illustrates the Lie algebra $L\wedge L$ is isomorphic to a factor of the free Lie algebra $F^2$.
\begin{lem}\cite[Theorem 2.10]{ni'}\label{k12}
Let $0\rightarrow R\rightarrow F \xrightarrow{\pi} L\rightarrow 0$ be a free presentation of a Lie algebra $ L $.
Then \begin{align*}\delta: L\wedge L&\rightarrow F^2/[R,F]\\
x\wedge y&\mapsto [\tilde{x},\tilde{y}]+[R,F] \end{align*} is an isomorphism, in where $ \pi(\tilde{x}+R)=x $ and $ \pi(\tilde{y}+R)=y.$
\end{lem}
The next result is extract from the works of  Batten,  Moneyhun and Stitzinger (1996).
\begin{lem}\cite[Lemma 1]{ba1}\label{llll}
Let $L$ be a  Lie algebra such that $ \dim L/Z(L)=n.$ Then  $ \dim L^2\leq \dfrac{1}{2}n(n-1).$
\end{lem}

\section{Main result}
In this section, after examining certain upper bounds for $L\wedge L$ and $\mathcal{M}(L)$, we investigate the numerical inequality on the dimension
$\mathcal{M}(L).$  Then we classify all
nilpotent Lie algebras that attains the upper bound Theorem \ref{5}. They are exactly nilpotent Lie algebras of class two. Moreover, for nilpotent Lie algebras of class at least $3,$ we also improve the bound Theorem \ref{5}.

First, we begin with the following result for a Lie algebra, similar to the result of Blackburn for the group theory  case \cite{black burn}.

Let $0\rightarrow R\rightarrow F \xrightarrow{\pi} L\rightarrow 0$ be a free presentation of a Lie algebra $ L $. Then

\begin{thm}\label{lkk}
Let $ L $ be a finite dimensional nilpotent non-abelian Lie algebra of class two. Then  \[0\rightarrow \ker g \rightarrow L^2\otimes_{mod} L^{(ab)} \xrightarrow{g} \mathcal{M}(L)\rightarrow
\mathcal{M}(L^{(ab)})\rightarrow L^2 \rightarrow 0\]
is exact, in where
\[g: x\otimes (z+L^2)\in L^2\otimes_{mod}L^{(ab)}\mapsto [\tilde{x},\tilde{z}]+[R,F]\in \mathcal{M}(L)=R\cap F^2/[R,F],\]
 $ \pi(\tilde{x}+R)=x $ and $ \pi(\tilde{z}+R)=z.$
Moreover,
$ K=\langle [x,y]\otimes z+L^2 +
[z,x]\otimes y+L^2+[y,z]\otimes x+L^2|x,y,z\in L\rangle \subseteq \ker g.$
\end{thm}
\begin{proof}
By using  \cite[Lemma 1.2]{aras} for $c=1$, we have the following exact sequence
\[0\rightarrow \ker g\rightarrow L^2\otimes_{mod} L^{(ab)} \xrightarrow{g} \mathcal{M}(L)\rightarrow
\mathcal{M}(L^{(ab)})\rightarrow L^2 \rightarrow 0\]
in where
\[g:x\otimes (z+L^2)\in L^2\otimes_{mod} L^{(ab)}\rightarrow [\tilde{x},\tilde{z}]+[R,F]\in\mathcal{M}(L)=R\cap F^2/[R,F],\]
$ \pi(\tilde{x}+R)=x $ and $ \pi(\tilde{z}+R)=z.$
Putting $ K=\langle [x,y]\otimes z+L^2 +
[z,x]\otimes y+L^2+[y,z]\otimes x+L^2|x,y,z\in L\rangle.$
By using the Jacobi identities, \begin{align*}g ([x,y]\otimes z +
[z,x]\otimes y+[y,z]\otimes x)=
[[\tilde{x},\tilde{y}], \tilde{z}] +
[[\tilde{z},\tilde{x}], \tilde{y}]+[[\tilde{y},\tilde{z}],\tilde{x}]+[R,F]=0.
 \end{align*}
 Thus $ K\subseteq \ker g.$
\end{proof}
The following two theorems are similar to the results of Ellis in \cite{el} and Hauti in \cite{H} for the case of group theory.
\begin{thm}\label{151}
Let $L$ be a Lie algebra. Then
\begin{itemize}
\item[$(i)$]$L^i/L^{i+1}\wedge L/L^{i+1}\cong L^i/L^{i+1}\otimes_{mod} L/L^2$  for all $i\geq 2.$

\item[$(ii)$] The natural sequence $L^{i+1}\wedge L ~~\xrightarrow{\alpha_{i+1}} L^{i}\wedge L\xrightarrow{\eta_i} L^{i}/L^{i+1} \otimes_{mod} L/L^{2}\rightarrow 0$ is exact  for all $ i\geq 2.$
\end{itemize}
 \end{thm}
\begin{proof}
\begin{itemize}
\item[$(i)$]
Since $ L^i/L^{i+1}\subseteq Z(L/L^{i+1})\cap L^2/L^{i+1},$
Lemma \ref{943} implies $ L^i/L^{i+1}\wedge L/L^{i+1}\cong L^i/L^{i+1}\otimes_{mod} L/L^2$ for all $ i\geq 2.$

\item[$(ii)$]The result follows from Proposition \ref{lk} and part $(i).$
\end{itemize}
\end{proof}

\begin{prop}\label{r}
Let $L$ be a Lie algebra. Then
\begin{itemize}
\item[$(i)$]
the map $\gamma_L:  L^{(ab)}\otimes_{mod}  L^{(ab)}\otimes_{mod} L^{(ab)}  \rightarrow L^2/L^3\otimes_{mod} L/L^2$ given by
\[(x+L^2)\otimes (y+L^2) \otimes(z+L^2)\mapsto ([x,y]+L^3\otimes z+L^2 )+
([z,x]+L^3\otimes y+L^2)+([y,z]+L^3\otimes x+L^2)
\]
is a Lie homomorphism. If any two element of the set $\{ x,y,z\} $ are linearly dependent. Then $\gamma_L(x+L^2 \otimes y+L^2 \otimes z+L^2)=0.$
\item[$(ii)$]
Define the map
\begin{align*}
&\gamma'_2: (L/Z(L))^{(ab)}\otimes_{mod}  (L/Z(L))^{(ab)}\otimes_{mod}  (L/Z(L))^{(ab)} \rightarrow L^2/L^3\otimes_{mod}  (L/Z(L))^{(ab)}\\& \big{(}x+(L^2+Z(L)) \big{)}\otimes \big{(}y+(L^2+Z(L)) \big{)} \otimes \big{(}z+(L^2+Z(L))\big{)}\mapsto\\& \big{(}[x,y]+L^3\otimes z+(L^2+Z(L)) \big{)}+
\big{ (}[z,x]+L^3\otimes y+(L^2+Z(L))\big{)}+\\&\big{(}[y,z]+L^3\otimes x+(L^2+Z(L))\big{)}.\end{align*}
Then $\gamma'_2$
is a Lie homomorphism. Moreover,  if any two element of the set $\{ x,y,z\} $ are linearly dependent, then \[\gamma'_{2}(x+(L^2+Z(L))\otimes y+(L^2+Z(L))\otimes z+(L^2+Z(L)))=0.\]
\item[$(iii)$]The map
\begin{align*}
\gamma'_3: &(L/Z(L))^{(ab)}\otimes_{mod}  (L/Z(L))^{(ab)}\otimes_{mod}  (L/Z(L))^{(ab)}\otimes_{mod}  (L/Z(L))^{(ab)}\\& \rightarrow  L^3\otimes_{mod}  (L/Z(L))^{(ab)}~\text{given by}~\\&\big{(}x+(L^2+Z(L)) \big{)}\otimes \big{(}y+(L^2+Z(L)) \big{)} \otimes \big{(}z+(L^2+Z(L))\big{)}+ \big{(}w+(L^2+Z(L)) \big{)}\\
&\mapsto\big{(}[[x,y],z]\otimes w+(L^2+Z(L)) \big{)}+
\big{ (}[w,[x,y]]\otimes z+(L^2+Z(L))\big{)}+\\&\big{(}[[z,w],x]\otimes y+(L^2+Z(L))\big{)}+\big{(}[y,[z,w]]\otimes x+(L^2+Z(L))\big{)}
\end{align*}
is a Lie homomorphism.
\end{itemize}
\end{prop}
\begin{proof}
Clearly, $ \gamma_L $ is a Lie homomorphism.  Let $ x=\beta y,$ for a scalar $ \beta $. We claim that  $ \gamma_L(x+L^2 \otimes y+L^2 \otimes z+L^2)=0.$ Since $ [x,y]=0,$ we have \begin{align*}&\gamma_L(x+L^2 \otimes y+L^2 \otimes z+L^2)=\\&([z,x]+L^3\otimes y+L^2)+([y,z]+L^3\otimes x+L^2)=\\&
([z,\beta y]+L^3\otimes y+L^2)+([y,z]+L^3\otimes \beta y+L^2)=\\&\beta \big{(}([z,y]+L^3\otimes y+L^2)+([y,z]+L^3\otimes y+L^2 )\big{)}=\\&\beta \big{(}([z,y]+L^3\otimes y+L^2)-([z,y]+L^3\otimes y+L^2) \big{)}=0.\end{align*}
Thus $\gamma_{L}(x+L^2\otimes y+L^2\otimes z+L^2)=0.$
The cases $(ii)  $ and $(iii)  $ obtained by a similar way.
\end{proof}
The following preliminary result will also play an important role in the next.
\begin{lem}\label{kp} The following natural sequence of abelian Lie algebras
\[ L^{i}/L^{i+1} \otimes_{mod} \big{(}(Z(L)+L^{2})/L^{2}\big{)}\xrightarrow{ \tau_i'} L^{i}/L^{i+1} \otimes_{mod} L/L^{2} \xrightarrow{ \delta_i}
L^{i}/L^{i+1} \otimes_{mod} \big{(}L/(Z(L)+L^{2})\big{)}\rightarrow 0\] is exact for all $ i\geq 2.$
\end{lem}
\begin{proof}
The proof is straightforward.
\end{proof}

We now make an observation to the
 Proposition \ref{lk} and Theorem \ref{151} $(ii),$  we can easily check that the following maps are homomorphisms

$\eta_{1}:l\wedge l_1 \in L\wedge L \mapsto (l+L^2\wedge l_1+L^2)\in L/L^{2} \wedge L/L^{2},$

$\eta_{i}:x\wedge y\in L^{i}\wedge L \mapsto (x+L^{i+1}\otimes y+L^2 )\in L^{i}/L^{i+1} \otimes_{mod} L/L^{2},$

$\alpha_{i}:x\wedge z \in L^{i}\wedge L \mapsto (x \wedge z)\in  L^{i-1}\wedge L$,
and $\kappa_i':(x_1\wedge z_1 )\in L\wedge L^{i-1} \mapsto [x_1,z_1]\in L^{i}$
for all $ i\geq 2.$
 $~\text{Put}~K=\langle [x,y]\wedge z +
[z,x]\wedge y+[y,z]\wedge x|x,y,z\in L\rangle \subseteq L^2\wedge L ~\text{and}$ \[K_1=\langle[[x,y],z]\wedge  w+
[w,[x,y]]\wedge z+[[z,w],x]\wedge y+[y,[z,w]]\wedge x|w,x,y,z\in L\rangle \subseteq L^3\wedge L.\]
By remembering the last homomorphisms, we are now ready to prove
\begin{prop}\label{gg}
Consider the canonical homomorphisms \[\tau_i: L^{i}\wedge Z(L)\rightarrow L^{i}\wedge L,\]  \[\tau_i': L^{i}/L^{i+1} \otimes_{mod} \big{(}(Z(L)+L^{2})/L^{2}\big{)}\rightarrow L^{i}/L^{i+1} \otimes_{mod} L/L^{2}\]  and \[\delta_i: L^{i}/L^{i+1} \otimes_{mod} L/L^{2}\twoheadrightarrow L^{i}/L^{i+1} \otimes_{mod} \big{(}L/(Z(L)+L^{2})\big{)}\] for all $ i\geq 2$. Then
$ \alpha_i(\mathrm{Im}\tau_i)=0_{L^{i-1}\wedge L}, $ $ \eta_{i_{\big|_{\mathrm{Im}\tau_i}}}(\mathrm{Im} \tau_i)=\mathrm{Im} \tau_i'=\ker \delta_i,$ $\mathrm{Im}  \tau_2\subseteq K\subseteq \ker \alpha_2,$ $ \dim \tau'_2\leq\dim \mathrm{Im} \gamma_L\leq \dim K \leq \dim \ker \alpha_2, $  $\mathrm{Im} ( \delta_{2_{\big|_{Im \gamma_L}}})=\mathrm{Im} \gamma_2'$ and $\ker( \delta_{2_{\big|_{\mathrm{Im} \gamma_L}}})=\mathrm{Im} \tau_2'$
for all $ i\geq 2.$ Moreover, $K_1\subseteq \ker \alpha_3$ and $\dim \mathrm{Im} \gamma_3'\leq \dim K_1\leq \dim \ker \alpha_3.$
\end{prop}
\begin{proof}
We claim that
$\alpha_i(x\wedge z)=0_{L^{i-1}\wedge L} $ for all $ z\in Z(L) $
and $ x\in L^i$ for all $i\geq 2$. There exists $ \bar{x} \in L\wedge L^{i-1}  $ such that
$\kappa_i'(\bar{x})=x.$ We may assume that  $\bar{x}=\sum_{t=1}^r \beta_t (l_t'\wedge l_t ),$  where $ l_t\in L^{i-1}, l_t'\in L$  and $\beta_t$ is scalar.  Then \[x=\kappa_i'(\bar{x})=\kappa_i'(\sum_{t=1}^r \beta_t (l_t'\wedge l_t ))=\sum_{t=1}^r \beta_t ([l_t',l_t] ).\]
Therefore
\begin{align*}\alpha_i(x\wedge z)= x\wedge z&=(\sum_{t=1}^r \beta_t [l_t',l_t] )\wedge z= \sum_{t=1}^r \beta_t ([l_t',l_t] \wedge z)\\&=\sum_{t=1}^r \beta_t (l_t' \wedge [l_t,z]-l_t \wedge [l_t',z])=0_{L^{i-1 } \wedge L}.\end{align*}
Thus $ \alpha_i(\text{Im} \tau_i)=0_{L^{i-1}\wedge L}. $
Consider the restriction of homomorphism $ \eta_i $ to $ \text{Im} \tau_i$ as follows
\[\eta_{i_{\big|_{\text{Im} \tau_i}}}:x\wedge z \in \text{Im} \tau_i \mapsto (x+L^{i+1}\otimes z+L^2)\in L^{i}/L^{i+1} \otimes_{mod} L/L^{2}.\]
Obviously, $\text{Im} ( \eta_{i_{\big|_{\text{Im} \tau_i}}})=\text{Im} \tau_i'.$
By invoking
 Lemma \ref{k12} , $\delta: L\wedge L\rightarrow F^2/[R,F]$ is an isomorphism and $\delta ([x,y]\wedge z +
[z,x]\wedge y+[y,z]\wedge x)=
[[\tilde{x},\tilde{y}], \tilde{z}] +
[[\tilde{z},\tilde{x}], \tilde{y}]+[[\tilde{y},\tilde{z}],\tilde{x}]+[R,F]=0.$ Using the Jacobi identities, we have $[[\tilde{x},\tilde{y}], \tilde{z}] +
[[\tilde{z},\tilde{x}], \tilde{y}]+[[\tilde{y},\tilde{z}],\tilde{x}]+[R,F]=0 $   for $x,y,z\in L.$ Thus
\[[x,y]\wedge z +
[z,x]\wedge y+[y,z]\wedge x=0_{L\wedge L}~\text{for all}~x,y,z\in L.\]
 Hence $\alpha_2(K)=0$ and so $K\subseteq \ker \alpha_2.$ By the restriction of  $\eta_2$ to $K$ and Proposition \ref{r} $(i),$ we have  $\eta_2(K)=\text{Im} \gamma_L$. Thus
$\dim \text{Im} \gamma_L\leq \dim K\leq \dim \ker \alpha_2.$
Now we show that $\text{Im} \tau'_2 \subseteq  \text{Im} \gamma_L. $
Let $ z\in Z(L) $ and $ d=\sum_{i=1}^t \alpha_i [x_i,y_i] $ for $x_i,y_i\in L.$
Since \begin{align*}&\big{(}[x_i,y_i]+L^3 \otimes z+L^2\big{)}+
\big{ (}[z,x_i]+L^3\otimes y_i+L^2\big{)}+\big{(}[y_i,z]+L^3\otimes x_i+L^2\big{)}=\\&\big{(}[x_i,y_i]+L^3 \otimes z+L^2\big{)}~\text{for all}~ 1\leq i\leq t, \end{align*}  we obtain $\big{(}[x_i,y_i]+L^3 \otimes z+L^2\big{)}\in \text{Im} \gamma_L $.  Thus
\begin{align*}
d+L^3\otimes  z+L^2&=(\sum_{i=1}^t \alpha_i [x_i,y_i]+L^3)\otimes z+L^2=\\&\sum_{i=1}^t (\alpha_i [x_i,y_i]+L^3\otimes z+L^2)\in \text{Im} \gamma_L.
\end{align*}
Therefore $\text{Im} \tau'_2 \subseteq  \text{Im} \gamma_L. $ Similarly $\text{Im} \tau_2\subseteq K\subseteq \ker \alpha_2.$
By the restriction of  $\delta_2$ to $\text{Im} \gamma_L,$ we have
\begin{align*}
&\delta_{2_{\big|_{\text{Im} \gamma_L}}}:\text{Im} \gamma_L \rightarrow L^{2}/L^{3} \otimes_{mod} L/(L^{2}+Z(L))~ \text{given by}\\
&([x,y]+L^3\otimes z+L^2 )+
([z,x]+L^3\otimes y+L^2)+([y,z]+L^3\otimes x+L^2)\mapsto \\&\big{(}[x,y]+L^3\otimes z+(L^2+Z(L)) \big{)}+
\big{ (}[z,x]+L^3\otimes y+(L^2+Z(L))\big{)}+\\&\big{(}[y,z]+L^3\otimes x+(L^2+Z(L))\big{).}
\end{align*}
Obviously, $\text{Im} ( \delta_{2_{\big|_{\text{Im} \gamma_L}}})=\text{Im} \gamma_2'$ and $\ker( \delta_{2_{\big|_{\text{Im} \gamma_L}}})=\text{Im} \tau_2'.$
 We show that $K_1\subseteq \ker \alpha_3$ and $\dim \text{Im} \gamma_3'\leq \dim K_1\leq \dim \ker \alpha_3.$
Since
\begin{align*}&\alpha_3([[x,y],z]\wedge w+[w,[x,y]]\wedge z+[[z,w],x]\wedge y+[y,[z,w]]\wedge x)=\\& [[x,y],z]\wedge w+[w,[x,y]]\wedge z+[[z,w],x]\wedge y+[y,[z,w]]\wedge x=\\&
[x,y]\wedge [z,w]-z\wedge [[x,y],w]+[w,[x,y]]\wedge z+[z,w]\wedge[x, y]\\&-x\wedge [[z,w],y]+[y,[z,w]]\wedge x=\\&
[x,y]\wedge [z,w]+z\wedge [w,[x,y]]+[w,[x,y]]\wedge z+[z,w]\wedge[x, y]\\&+x\wedge [y,[z,w]]+[y,[z,w]]\wedge x= 0_{L^2\wedge L},
\end{align*}
we have $K_1\subseteq \ker \alpha_3.$
Similarly, Proposition \ref{r} $(iii)$ implies $\eta_3(K_1)=\text{Im} \gamma_3'$
 and so
$\dim \text{Im} \gamma_3'\leq \dim K_1\leq \dim \ker \alpha_3,$ as required.
\end{proof}

\begin{thm}\label{25}
Let $ L $ be a finite dimensional nilpotent non-abelian Lie algebra of class $c.$  Then
 \begin{align*}\dim L\wedge L+\dim \mathrm{Im} \gamma_L &\leq\dim L\wedge L+\sum_{i=2}^c \dim \ker \alpha_i\\&=\dim L/L^2\wedge L/L^2+\sum_{i=2}^c \dim (L^i/L^{i+1}\otimes_{mod} L/L^2).\end{align*}

\end{thm}
\begin{proof}
By using Proposition \ref{lk} and Theorem \ref{151} $(ii),$  the following two sequences
\begin{equation}\label{1e}
L^{2}\wedge L ~~\xrightarrow{\alpha_2} L\wedge L~~\xrightarrow{\eta_1}  L/L^{2} \wedge L/L^{2}\rightarrow 0
\end{equation} and
\begin{equation}\label{2e}
L^{i+1}\wedge L ~~\xrightarrow{\alpha_{i+1}} L^{i}\wedge L~~\xrightarrow{\eta_i}  L^{i}/L^{i+1} \otimes_{mod} L/L^{2}\rightarrow 0,
\end{equation} are exact
for all $ i\geq 2.$

Using \ref{1e} and \ref{2e}, we have
\begin{equation}\label{eq3} \dim L\wedge L+\sum_{i=2}^c \dim \ker \alpha_i=\dim (L/L^2\wedge L/L^2)+\sum_{i=2}^c (\dim L^i/L^{i+1}\otimes_{mod} L/L^2).\end{equation}

Now Proposition \ref{gg} implies $ \dim \text{Im} \gamma_L\leq \dim \ker \alpha_2.$ Hence
\ref{eq3} deduces that \begin{align*} \dim L\wedge L+\dim \text{Im} \gamma_L &\leq  \dim L\wedge L+\dim \ker \alpha_2 \\&\leq
\dim L\wedge L+\sum_{i=2}^c \dim \ker \alpha_i=\\&\dim (L/L^2\wedge L/L^2)+\sum_{i=2}^c \dim (L^i/L^{i+1}\otimes_{mod} L/L^2),\end{align*}
as required.
\end{proof}

\begin{thm}\label{j}
Let $ L $ be a finite dimensional nilpotent non-abelian Lie algebra of class $c.$  Then
 \begin{align*}\dim L\wedge L+\dim \mathrm{Im} \gamma_2' &\leq\dim L\wedge L+\sum_{i=2}^c \dim \ker \alpha_i\\&=\dim (L/L^2\wedge L/L^2)+\sum_{i=2}^c \dim (L^i/L^{i+1}\otimes_{mod} (L/Z(L))^{(ab)}).\end{align*}
\end{thm}
\begin{proof}
By Proposition \ref{gg}, we have $\dim \text{Im} \gamma_{L}\leq \dim \ker \alpha_2.$ Thus
\begin{align*}
&\dim L\wedge L+\dim \text{Im} \gamma_{L}-\dim \text{Im} \tau_2'\leq\\& \dim L\wedge L+\dim \ker \alpha_2-\dim \text{Im} \tau_2'.\end{align*}
 Proposition  \ref{gg} implies $ \dim \text{Im}\tau_i'\leq  \dim \text{Im}\tau_i \leq \dim \ker \alpha_i$ for all $i\ge 2$. Hence
 \begin{align*} &\dim L\wedge L+\dim \ker \alpha_2-\dim \text{Im} \tau_2'\\\leq& \dim L\wedge L+\sum_{i=2}^c (\dim \ker \alpha_i-\dim \text{Im} \tau_i').
 \end{align*}Hence $$\dim L\wedge L+\dim \text{Im} \gamma_{L}-\dim \text{Im} \tau_2'\leq \dim L\wedge L+\sum_{i=2}^c (\dim \ker \alpha_i-\dim \text{Im} \tau_i').$$
  On the other hand, with the aid of Theorem \ref{25}, we have
\begin{align*}
&\dim L\wedge L+\sum_{i=2}^c (\dim \ker \alpha_i-\dim \text{Im} \tau_i')\\&=\dim (L/L^2\wedge L/L^2)+\sum_{i=2}^c \big(\dim (L^i/L^{i+1}\otimes_{mod} L/L^2)-\dim \text{Im} \tau_i'\big).\end{align*}

Now Lemma \ref{kp}  implies
\[ \dim (L^i/L^{i+1}\otimes_{mod} L/L^2)=\dim (L^i/L^{i+1}\otimes_{mod} (L/Z(L))^{(ab)})+\dim \text{Im} \tau_i',\] for $i\geq 2.$
Thus
\begin{align*}
&\dim (L/L^2\wedge L/L^2)+\sum_{i=2}^c \big(\dim (L^i/L^{i+1}\otimes_{mod} L/L^2)-\dim \text{Im} \tau_i'\big)=\\&\dim (L/L^2\wedge L/L^2)+\sum_{i=2}^c \dim (L^i/L^{i+1}\otimes_{mod} (L/Z(L))^{(ab)}).\end{align*}
By the fact that $\dim\text{Im} \gamma_L-\dim \text{Im} \tau_2'=\dim \text{Im} \gamma_2'$ in the proof of Proposition  \ref{gg}, we have
\begin{align*}
\dim L\wedge L+\dim \text{Im} \gamma_2' &\leq \dim L/L^2\wedge L/L^2+\sum_{i=2}^c \dim (L^i/L^{i+1}\otimes_{mod}  (L/Z(L))^{(ab)}),\end{align*} as required.
\end{proof}
Recall that a Lie algebra $L$ is called stem provided that $Z(L)\subseteq L^2$.
\begin{thm}\label{mr}
Let $ L $ be a finite dimensional nilpotent stem Lie algebra of class $3.$  Then,  we have
\begin{align*}&\dim L\wedge L+\dim \mathrm{Im} \gamma_2' +\dim \mathrm{Im} \gamma_3' \\&\leq \dim (L/L^2\wedge L/L^2)+ \dim (L^2/L^{3}\otimes_{mod} L^{(ab)})+\dim (L^3\otimes_{mod} L^{(ab)}).\end{align*}
\end{thm}
\begin{proof}

Proposition \ref{gg} implies $\dim \text{Im} \gamma_3'\leq \dim K_1\leq \dim \ker \alpha_3.$ But   $\dim\ker \alpha_3\leq \dim L^{3}\wedge L$ and Lemma \ref{943} implies that  $\dim L^{3}\wedge L= \dim L^{3}\otimes_{mod} L^{(ab)}$.  Now the result directly obtained from Theorem \ref{j}.
\end{proof}
Looking the proof of \cite[Theorem 3.1]{ni}, we have
\begin{prop}\label{d}
Let $ L $ be a nilpotent Lie algebra of dimension $ n $ such that $ \dim L^2=1 $. Then $L\cong H(m)\oplus A(n-2m-1).$
\begin{itemize}
\item[$(i)$] If $ m=1,$ then $ \dim\mathcal{M}(L)=\dfrac{1}{2}(n-1)(n-2)+1.$
\item[$(ii)$] If $ m\geq 2,$ then $ \dim\mathcal{M}(L)=\dfrac{1}{2}(n-1)(n-2)-1.$
\end{itemize}

\end{prop}

\begin{lem}\label{ller}
Let $L$ be a  finite dimensional  nilpotent Lie algebra and $\dim L/L^2=1.$ Then $\dim L=1.$
\end{lem}
\begin{proof}
Since $\dim L/L^2=1$ and the Frattini subalgebra $L$ is equal to $L^2$, we have $\dim L=1,$ as required.
\end{proof}
The following theorem improves Theorem \ref{5}.
\begin{thm}\label{51} Let $L$ be an $n$-dimensional non-abelian nilpotent Lie algebra of class $c$ with the derived subalgebra of dimension $ m$ and $ t=\dim (Z(L)/Z(L)\cap L^2).$  Then \[
\dim \mathcal{M}(L)\leq \dfrac{1}{2}(n+m-2)(n-m-1)-t(m-1)+1.\]
In particular, when $m = 1$ the bound is attained if and only if $L=H(1)\oplus A(n-3).$
\end{thm}
\begin{proof}
If $ m=1,$ then the result holds by Proposition \ref{d}.
Thus we may assume that $m > 1.$ By invoking Lemmas \ref{2} and \ref{j1}, $ \dim L/L^2\wedge L/L^2= \dim \mathcal{M}(L/L^2)=\dfrac{1}{2}(n - m )(n - m - 1).$
Therefore Theorem \ref{j} implies
\begin{align*}\dim L\wedge L+\dim \text{Im} \gamma_2'&\leq\dim L\wedge L+\sum_{i=2}^c \dim \ker \alpha_i\\&=\dim (L/L^2\wedge L/L^2)+\sum_{i=2}^c \dim (L^i/L^{i+1}\otimes_{mod} (L/Z(L))^{(ab)})\\=&\dfrac{1}{2}(n - m )(n - m - 1)+\sum_{i=2}^c \dim (L^i/L^{i+1}\otimes_{mod} (L/Z(L))^{(ab)}).\end{align*}
On the other hand, \begin{align*}\sum_{i=2}^c \dim (L^i/L^{i+1}\otimes_{mod} (L/Z(L))^{(ab)})&=\dim \big{(}(\bigoplus_{i=2}^c  L^i/L^{i+1})\otimes_{mod} (L/Z(L))^{(ab)}\big{)}\\& = \dim L^2 \dim (L/Z(L))^{(ab)}=m(n-m-t).
\end{align*}
Thus
\begin{align*}
\dim L\wedge L+\dim \text{Im} \gamma_2'\leq \dfrac{1}{2}(n - m )(n - m - 1)+m(n-m-t).
\end{align*}
Now we are going to obtain a lower bound for the dimension of $\text{Im} \gamma_2'.$

If $\dim (L/Z(L))^{(ab)}= 1,$ by using Lemma \ref{ller}, we have  $\dim(L/Z(L))=1, $  and so $L$ is abelian that is impossible. Hence  $ \dim (L/Z(L))^{(ab)}\geq 2.$ Set $  d=\dim (L/(Z(L)+L^2)).$
We claim that $ d-2 \leq \dim \text{Im} \gamma_2'$ for $ d\geq 2.$
If $d= \dim (L/(Z(L)+L^2))=\dim (L/Z(L))^{(ab)}=2,$ then $ \text{Im} \gamma_2'=0,$ by Proposition \ref{r} $(ii).$ Thus $ d-2=0=\dim \text{Im} \gamma_2'.$
Suppose that $ d=\dim (L/(Z(L)+L^2)) \geq 3.$ We can choose a basis
\[ \{x_1+Z(L)+L^2 ,\ldots, x_d+Z(L)+L^2\} \] for $L/(Z(L)+L^2)  $
such that $ [x_1,x_2]+L^3 $ is non-trivial in $ L^2/L^3.$
We claim that all elements of the set \[B=\{ \gamma_2'\big{(}(x_1+(Z(L)+L^2))\otimes (x_2+(Z(L)+L^2))\otimes (x_i+(Z(L)+L^2))\big{)}|3\leq i\leq d\}\] are linearly  independent. Since
\[ L^2/L^3 \otimes_{mod} (L/Z(L))^{(ab)}\cong \bigoplus_{i=1}^d  \big(L^2/L^3 \otimes_{mod} \langle x_i+(Z(L)+L^2) \rangle\big)\]
and for $i\geq 3$ \begin{align*}
&\gamma_2'\big{(}x_1+(Z(L)+L^2)\otimes x_2+(Z(L)+L^2)\otimes x_i+(Z(L)+L^2 )\big{)}\\&=  [x_1,x_2]+L^3\otimes x_i+(L^2 +Z(L))+
[x_i,x_1]+L^3\otimes x_2+(L^2+Z(L))\\&+ [x_2,x_i]+L^3\otimes x_1+(L^2+Z(L)),
\end{align*}
we have \begin{align*}
&\gamma_2'\big{(}x_1+(Z(L)+L^2)\otimes x_2+(Z(L)+L^2)\otimes x_i+(Z(L)+L^2 )\big{)}\in \\& \langle [x_1,x_2]+L^3\otimes x_i+(L^2 +Z(L))\rangle \oplus
\langle [x_i,x_1]+L^3\otimes x_2+(L^2+Z(L))\rangle \oplus \\&\langle [x_2,x_i]+L^3\otimes x_1+(L^2+Z(L))\rangle.
\end{align*}
Since $ [x_1,x_2]\notin L^3 $ and $ x_i\notin L^2+Z(L),$
 $[x_1,x_2]+L^3\otimes x_i+(L^2 +Z(L)) $ is non-trivial element in $ L^2/L^3 \otimes_{mod} \langle x_i+(Z(L)+L^2) \rangle,$
$ \gamma_2'\big{(}x_1+(Z(L)+L^2)\otimes x_2+(Z(L)+L^2)\otimes x_i+(Z(L)+L^2 )\big{)}\neq 0.$
Hence all elements of  \begin{align*} \gamma_2'( B)&= \{ [x_1,x_2]+L^3\otimes x_i+(L^2 +Z(L))+
[x_i,x_1]+L^3\otimes x_2+(L^2+Z(L))\\&+ [x_2,x_i]+L^3\otimes x_1+(L^2+Z(L))|3\leq i\leq d\}
\end{align*} are linearly independent  and so
$ d-2 \leq \dim \text{Im} \gamma_2'.$ By Lemma \ref{j1}, we have
\begin{align*}
\dim \mathcal{M}(L)+\dim L^2+d-2 \leq \dfrac{1}{2}(n - m )(n - m - 1)+ m(n-m-t).
\end{align*}
Since $d=n-m-t,  $ we have
\begin{align*}
&\dim \mathcal{M}(L)+m+(n-m-t-2) \leq \dfrac{1}{2}(n - m )(n - m - 1)+ m(n-m-t)=\\& \dfrac{1}{2}(n - m )(n - m - 1)+ (m-1)(n - m - 1)-(m-1)(n - m - 1)+m(n-m-t)=\\&\dfrac{1}{2} (n + m - 2)(n - m - 1)-m(n-m)+(n-m-1)+m+m(n-m-t)=\\&\dfrac{1}{2} (n + m - 2)(n - m - 1)+(n-m-1)+m-mt=\\&\dfrac{1}{2} (n + m - 2)(n - m - 1)+n-1-mt.
\end{align*}
Thus
\begin{align*}
\dim \mathcal{M}(L) &\leq \dfrac{1}{2} (n + m - 2)(n - m - 1)+n-1-mt-m-(n-m-t-2)\\& \dfrac{1}{2} (n + m - 2)(n - m - 1)-mt-1+t+2=\\&\dfrac{1}{2} (n + m - 2)(n - m - 1)-t(m-1)+1,\end{align*}
as claimed. If $ m=1,$ then the converse holds by Theorem \ref{5}.
\end{proof}

According to the notation and terminology of the classification of nilpotent Lie algebras of dimension at most 6 in \cite{cic}, let \[L_{5,8}=\langle x_1,\ldots, x_5 | [x_1, x_2] = x_4, [x_1, x_3] = x_5\rangle\] and \[L_{6,26}=\langle   x_1,\ldots, x_6 | [x_1, x_2] = x_4, [x_1, x_3] = x_5, [x_2, x_3] = x_6 \rangle.\]
Note that from the notation of \cite{ha, ni2}, $L_{5,8}$ is also denoted by $L(4, 5, 2, 4)$.

The following example shows that the upper bound of Theorem \ref{51} can be obtained.
\begin{ex}
Let $ L=L_{5,8}\oplus A(1).$ By Lemmas \ref{1kg} and \ref{f1},
we have
\[\dim \mathcal{M}(L)=\dim \mathcal{M}(L_{5,8})+\dim L_{5,8} /(L_{5,8})^2=6+3=9. \]  Since $n=6, m=2$ and $t=1,$
\[\dim \mathcal{M}(L)=\dfrac{1}{2} (n + m - 2)(n - m - 1)-t(m-1)+1=\dfrac{1}{2} 6(6-3)-1+1=9.\]
\end{ex}
The following lemma gives the structure of all $n$-dimensional nilpotent Lie algebra with the derived subalgebra of dimension two when the bound of Theorem \ref{5} is attained.
\begin{lem}\label{f1}
Let $L$ be an $n$-dimensional nilpotent Lie algebra with the derived subalgebra of dimension two. Then
$ \dim \mathcal{M}(L)=\dfrac{1}{2}n(n-3)+1 $ if and only if $ L\cong L_{5,8}.$
\end{lem}
\begin{proof}
The result follows from \cite[Theorem 3.9]{ni2}, since $m=2$.
\end{proof}
\begin{cor}\label{112pop}
There is no $n$-dimensional nilpotent Lie algebra of nilpotency class $c\geq 3 $ such that $ \dim \mathcal{M}(L)=\dfrac{1}{2}n(n-3)+1. $ \end{cor}
\begin{prop}\label{fjj}
Let $L$ be an $n$-dimensional non-abelian nilpotent Lie algebra  with the derived subalgebra of dimension $ m$ and $ \dim \mathcal{M}(L)=\dfrac{1}{2}(n+m-2)(n-m-1)+1.$ If $ m\geq 2,$
  then $ L$ is stem.
\end{prop}
\begin{proof}
Putting $t=\dim (Z(L)/Z(L)\cap L^2).$
 By Theorem \ref{51}, we have \begin{align*}\dim \mathcal{M}(L) &=\dfrac{1}{2}(n+m-2)(n-m-1)+1\\&<\dfrac{1}{2}(n+m-2)(n-m-1)+1-t(m-1). \end{align*}
Thus $ t=0$ and so $Z(L)\subseteq L^2,$ as required.
\end{proof}
\begin{prop}\label{fjj1}
Let $L$ be an $n$-dimensional nilpotent Lie algebra  with the derived subalgebra of dimension $ m$ and $ \dim \mathcal{M}(L)=\dfrac{1}{2}(n+m-2)(n-m-1)+1.$ If $ m\geq 2$ and
$K$ is an ideal  of dimension $1$
contained in $ Z(L).$  Then $L/K$ attains the bound Theorem \ref{5}, that means \[\dim \mathcal{M}(L/K) =\dim \mathcal{M}(L)-\dim L/L^2+\dim(L^2\cap K)=\dfrac{1}{2}(n + m - 4)(n - m - 1) + 1.\]
\end{prop}
\begin{proof}
By  Proposition \ref{fjj}, $L$ is stem.
Using Lemma \ref{d1}, we have
\begin{align*} \dim \mathcal{M}(L)+\dim(L^2\cap K) &\leq \dim \mathcal{M}(L/K) +\dim \mathcal{M}(K) +\dim(L/L^2 \otimes_{mod} K)\\
&\leq \dim \mathcal{M}(L/K)+\dim(L/L^2 \otimes_{mod} K).
\end{align*}
Thus
\begin{align*}
\dim \mathcal{M}(L)&\leq \dim \mathcal{M}(L/K)+\dim(L/L^2 \otimes_{mod} K)-\dim (L^2\cap K)\\&\leq
\dfrac{1}{2}(n + m - 4)(n - m - 1) + 1 + n - m - 1\\
=&\dfrac{1}{2}
(n + m - 2)(n - m - 1) + 1=\dim \mathcal{M}(L).
\end{align*}
Therefore \[\dim\mathcal{M}(L/K)=\dim \mathcal{M}(L)-\dim L/L^2+\dim(L^2\cap K)=\dfrac{1}{2}(n + m - 4)(n - m - 1) + 1,\]
as required.
\end{proof}

\begin{prop}\label{fjj19}
Let $L$ be an $n$-dimensional  nilpotent Lie algebra  with the derived subalgebra of dimension $ m$ and $ \dim \mathcal{M}(L)=\dfrac{1}{2}(n+m-2)(n-m-1)+1$. If $ m\geq 2$ and
$K$ is a non-zero  ideal  of dimension $k$
contained in $ Z(L)$. Then $L/K$ attains the bound Theorem \ref{5}, in the other words  \[\dim \mathcal{M}(L/K) =\dfrac{1}{2}(n + m - 2(k+1))(n - m - 1) + 1.\]
\end{prop}
\begin{proof}
By  Proposition \ref{fjj}, we have $L$ is stem. Let  $K$ be an ideal  of dimension $k$
contained in $ Z(L).$ We have $ \dim L/K=n-k $ and $ \dim (L/K)^2=m-k.$ We prove the result by induction on $k.$ If $k=1,$ then the result holds by  Proposition \ref{fjj1}. Now let $ k\geq 2.$ There exists an ideal $ K_1 $ in $ K $ with dimension $ k-1. $ Using the  hypothesis induction, we have \[\dim \mathcal{M}(L/K_1) =\dfrac{1}{2}(n + m - 2k)(n - m - 1) + 1.\]
Since $ K/K_1 $ is one dimensional ideal in $ Z(L/K_1),$
  Proposition \ref{fjj1} implies that \[\dim \mathcal{M}(L/K)=\dim \mathcal{M}(\dfrac{L/K_1}{K/K_1}) =\dfrac{1}{2}(n + m - 2(k+1))(n - m - 1) + 1.\]
And this completes the proof.
\end{proof}

\begin{prop}\label{f}
Let $L$ be an $n$-dimensional  nilpotent Lie algebra  with the  derived subalgebra of dimension $ m$ and $ \dim \mathcal{M}(L)=\dfrac{1}{2}(n+m-2)(n-m-1)+1.$ If $ m\geq 3$ and $ n-m\leq 3,$
  then  $ L$ is stem and $n- m=3.$
\end{prop}
\begin{proof}
By Proposition \ref{fjj},    $ L$ is stem. If $ n-m=1,$ then by Lemma \ref{ller}, $ L $ is abelian, which is impossible.
If $ n-m=2,$ then by Lemma \ref{35}, $ \dim \mathcal{M}(L)=m+1<m-2. $ So we have a contradiction again. Hence $ n-m=3.$
\end{proof}

\begin{defn}\cite[Definition 2.1]{ni4}\label{1pl}
A Lie algebra $ H $ is called generalized Heisenberg of rank $ n $ if $ H^2=Z(H) $ and $ \dim H^2=n.$
\end{defn}
\begin{thm}\label{f2}
Let $L$ be an $n$-dimensional nilpotent Lie algebra of class two with the derived subalgebra of dimension $ m$ and $ \dim \mathcal{M}(L) =\dfrac{1}{2}(n+m-2)(n-m-1)+1.$ If $ m\geq 3$  and $n-m\geq 3,$ then $ L $ is generalized Heisenberg of rank $3$ and dimension $ 6. $
\end{thm}
\begin{proof}
By Theorem  \ref{lkk}, we have $  \dim \mathcal{M}(L)+\dim \ker g-\dim L^2=\dim \mathcal{M}(L^{(ab)})+\dim L^{(ab)}\otimes_{mod} L^2,$  where
$\langle [x,y]\otimes z+L^2 +[z,x]\otimes y+L^2+[y,z]\otimes x+L^2|x,y,z\in L\rangle\subseteq \ker g.$ By using Proposition \ref{f},
 $ L $ is stem of class two and so $Z(L)=L^2$. Since $\dim \mathcal{M}(L)=\dfrac{1}{2}(n+m-2)(n-m-1)+1$ and
 $\dim L/L^2= n-m,$ we have $ \dim X=n-m-2.$

 If $d=\dim L/Z(L)= n-m \geq 4, $ then since $ \dim L^2\geq 3, $  without loss of generality, we can choose a basis
$ \{x_1+L^2 ,\ldots, x_d+L^2\} $ for $L/L^2  $
such that $ [x_1,x_2], $  $ [x_2,x_3] $ and $ [x_3,x_4] $  are  non-trivial in $ L^2.$ Thus
\[ L^2 \otimes_{mod} L/L^2\cong \bigoplus_{i=1}^d \big( L^2 \otimes_{mod} \langle x_i+L^2 \rangle\big).\] Hence all elements of \begin{align*} \{ [x_1,x_2]\otimes x_i+L^2 \oplus
[x_i,x_1]\otimes x_2+L^2\oplus [x_2,x_i]\otimes x_1+L^2,|3\leq i\leq d,i\neq 1,2 \}
\end{align*}
and
\begin{align*} \{ [x_2,x_3]\otimes x_i+L^2 \oplus
[x_i,x_2]\otimes x_3+L^2\oplus [x_3,x_i]\otimes x_2+L^2,|3\leq i\leq d,i\neq 1,2,3\}
\end{align*}
\begin{align*} \{ [x_3,x_4]\otimes x_i+L^2 \oplus
[x_i,x_3]\otimes x_4+L^2\oplus [x_4,x_i]\otimes x_3+L^2,|3\leq i\leq d,i\neq 2,3,4\}
\end{align*}
 are linearly independent and hence
$2( n-m-3)+n- m-2\leq \dim X.$ That is a contradiction.
Therefore $ n-m=3.$ Now  Lemma \ref{llll} implies    $\dim L^2\leq 3.$
Thus $n=6$ and $m=3.$ Hence $ L $ is a generalized Heisenberg Lie algebra of  rank $3.$ By looking the classification of nilpotent Lie algebras given in \cite{cic}, we should have $L\cong L_{6,26}$.
\end{proof}
 We use the method of Hardy and Stitzinger in \cite{ha} to compute the Schur multiplier of $L_{6,26}$ as below.
\begin{prop}\label{f4}
The Schur multiplier of Lie algebra $L_{6,26}  $  is an abelian Lie algebra of dimension $ 8.$
\end{prop}
\begin{proof}
Let $L =L_{6,26}=\langle x_1,\ldots,x_6\rangle$  with $[x_1, x_2] = x_4, [x_1, x_3] = x_5, [x_2, x_3] = x_6$
start with
\begin{align*}
&[x_1, x_2] = x_4+s_1, [x_1, x_3] = x_5+s_2,\\
&[x_1, x_4] = s_3, [x_1, x_5] = s_4,\\
&[x_1, x_6] = s_5, [x_2, x_3] = x_6+s_6,\\
&[x_2, x_4] = s_7,[x_2, x_5] = s_8,\\& [x_2, x_6] = s_9,[x_3, x_4] = s_{10},\\
&[x_3, x_5] = s_{11},[x_3, x_6] = s_{12},\\
&[x_4, x_5] = s_{13},[x_4, x_6] = s_{14},\\&[x_6, x_5] = s_{15},
\end{align*}
where $\{s_1, \ldots, s_{15}\}$ generate $\mathcal{M}(L)$. Using the Jacobi identities on all possible triples gives
\[s_5=s_8-s_{10}, s_{13}= s_{14}= s_{15}=0.\]

\begin{align*}
s_{13}=[[x_1,x_2],x_5]&=-([[x_5,x_1],x_2]+[[x_2,x_5],x_1])\\&=
[s_4,x_2]+[s_8,x_1]=0,
\end{align*}
\begin{align*}
s_{14}=[[x_1,x_2],x_6]&=-([[x_6,x_1],x_2]+[[x_2,x_6],x_1])\\&=
-([-s_5,x_2]+[s_9,x_1])=0,
\end{align*}
\begin{align*}
s_{15}=[[x_2,x_3],x_5]&=-([[x_5,x_2],x_3]+[[x_3,x_5],x_2])\\&=
-([-s_8,x_3]+[s_{11},x_2])=0,
\end{align*}
\begin{align*}
s_{5}=-[[x_2,x_3],x_1]&=([[x_1,x_2],x_3]+[[x_3,x_1],x_2])\\&=
([x_4,x_3]+[-x_5,x_2])=-s_{10}+s_8.
\end{align*}
Using another change
of basis in which we define $x_4'=x_4-s_1, $ $x_5'=x_5-s_2$ and  $x_{6}' =x_6-s_{6}$, we conclude that $\mathcal{M}(L)= \langle s_3,s_4,s_7,\ldots,s_{12}\rangle $ and so $ \dim \mathcal{M}(L)=8.$
\end{proof}

\begin{thm}\label{main}
Let $L$ be an $n$-dimensional nilpotent Lie algebra of class two with the derived subalgebra of dimension $ m.$ Then  $ \dim \mathcal{M}(L)=\dfrac{1}{2}(n+m-2)(n-m-1)+1$ if and only if $ L $ is isomorphic to the one of Lie algebras
$ H(1)\oplus A(n-3),~ L_{5,8}~\text{or}~L_{6,26}.$

\end{thm}

\begin{proof}
Let $\dim \mathcal{M}(L)=\dfrac{1}{2}(n+m-2)(n-m-1)+1.$ If $ m=1,$ then $ \dim \mathcal{M}(L)= \dfrac{1}{2}(n-1)(n-2)+1$ and Proposition \ref{d} shows that $L=H(1)\oplus A(n-3).$ If $ m= 2,$ then $ \dim \mathcal{M}(L)= \dfrac{1}{2}n(n-3)+1$ and Lemma \ref{f1} implies that $L\cong L_{5,8}.$ Let $ m\geq 3.$
 If $ n-m=1,$ then by Lemma \ref{ller}, $ L $ is abelian, which is impossible. If $ n-m=2,$ then by Lemma \ref{35}, $ \dim \mathcal{M}(L)=m+1<m-2, $ so we have a contradiction again. Thus  $ m\geq 3$ and $ n-m\geq 3.$  Using Theorem \ref{f2} we should have $L\cong L_{6,26}.$
The converse follows from using Propositions \ref{d}, \ref{f4} and Lemma \ref{f1}.
\end{proof}Remember that $H(1)$ is a Lie algebra with basis $x,y,z$ and the only non--zero multiplication between basis elements is given by $[x,y]=z$.
\begin{prop}\label{hem}
There is no $n$-dimensional nilpotent Lie algebra of nilpotency class $3$ and $\dim L^2=m$ such that  $\dim \mathcal{M}(L)=\dfrac{1}{2}(n+m-2)(n-m-1)+1.$
\end{prop}
\begin{proof}
By contrary, let $\dim \mathcal{M}(L)=\dfrac{1}{2}(n+m-2)(n-m-1)+1.$
 By Proposition \ref{fjj19},  $ \dim \mathcal{M}(L/Z(L)) $ also attains the bound of Theorem \ref{5}. Put $ \dim Z(L)=t.$ Hence by Theorem \ref{main},
$L/Z(L)$ is isomorphic to one of $ H(1)\oplus A(n-t-3),$ $L_{5,8}$ or $
L_{6,26}.$

Let $ L/Z(L)\cong H(1)\oplus A(n-t-3). $  By Proposition \ref{fjj}, $ Z(L)\subseteq L^2.$ Thus $ \dim L/L^2=\dim ( L/Z(L))^{(ab)}=n-t-1 $ and $ \dim L^2=t+1.$ If $ t=1, $ then $\dim L^2=2$ and so $\dim \mathcal{M}(L) = \dfrac{1}{ 2 }n(n-3)+1,$ which contradicts the result of Corollary \ref {112pop}.   Thus $ \dim L^2\geq 3.$ If $ \dim L/Z(L)=3, $ then $ n-m=2.$  Lemma \ref{35} implies $ \dim \mathcal{M}(L)\leq m. $ Now by the assumption, since $n=m+2,$  $ \dim \mathcal{M}(L)=m+1. $ Thus we have a contradiction. Therefore $\dim L/Z(L)\geq 4  $ and $ \dim L^2\geq 3.$ \\
Since $ L/Z(L)\cong H(1)\oplus A(n-t-3)$,  there exist two ideals $ I_1/Z(L) $ and $  I_2/Z(L) $ in $L/Z(L)  $ such that $I_1/Z(L)\cong H(1)  $ and $I_2/Z(L)\cong A(n-t-3).$
We may assume that $ I_1=\langle x_1,x_2,x_3,a_1,a_2,a_3|[x_1,x_2]=x_3+a_1,[x_2,x_3]=a_2,[x_1,x_3]=a_3,a_i\in Z(L)\rangle +Z(L)$ and $ I_2=\langle y_1,\ldots,y_{n-t-3}\rangle +Z(L),$ where for all $ 1\leq i\leq n-t-3$, we have $ y_i\notin Z(L).$
 
 We claim that $\dim \text{Im}\gamma_2'\geq n-m-2$ and $\dim \text{Im}\gamma_3'\geq 1.$
Since $ \dim L^2=\dim Z(L)+1,$  $ y_i\notin L^2$ for all $ 1\leq i\leq n-t-3.$ Also $ x_3\notin Z(L)$ implies $ [x_i,x_3]\neq 0,$  for some
$1 \leq i \leq 2. $ Using the same way in the proof of Theorem \ref{51}, we may obtain that
  $\dim \text{Im}\gamma_2'\geq  n-m-2.$ On the other hand, $ \gamma_3'(x_1+Z(L)\otimes x_2+Z(L) \otimes x_3+Z(L)\otimes y_1+Z(L)) $ is non-trivial and so also $\dim \text{Im}\gamma_3'\geq 1,$ as we claimed.
  
   Let $ \dim L^3 =m_1$ and using Lemma \ref{j1}, we have $\dim L\wedge L=\dim \mathcal{M}(L)+L^2$. Therefore Theorem \ref{mr} implies  \[\begin{array}{lcl}\dim \mathcal{M}(L)+m+n-m-2+1\leq \dim \mathcal{M}(L)+\dim L^2+\dim \text{Im} \gamma_2' +\dim \text{Im} \gamma_3'\leq\\ \dim (L/L^2\wedge L/L^2)+ \dim (L^2/L^{3}\otimes_{mod} L^{(ab)})+\dim (L^3\otimes_{mod} L^{(ab)})=\dfrac{1}{2}
 (n-m)\\(n-m-1)+(m-m_1)(n-m)+m_1(n-m)= \dfrac{1}{2}
 (n-m)(n-m-1)+m(n-m).\end{array}\]Thus \[\begin{array}{lcl}
 \\&\dim \mathcal{M}(L)\leq \dfrac{1}{2}
 (n-m)(n-m-1)+m(n-m)-m-(n-m-2)-1 =\\& \dfrac{1}{2}
 (n-m)(n-m-1)+(m-1)(n-m-1)= \dfrac{1}{2}(n+m-2)(n-m-1),
 \end{array}\]
  which is a contradiction.
Hence we should have $L/Z(L)\cong L_{5,8} $  or $L/Z(L)\cong L_{6,26}.$  By Proposition \ref{fjj}, $ Z(L)\subseteq L^2,$ and so
  $n-m=\dim L/L^2=3.$   Using the same way in the proof of Theorem \ref{51}, we may obtain that
  $\dim \text{Im}\gamma_2'\geq 1$ and $\dim \text{Im}\gamma_3\geq 1.$  Putting $ \dim L^3 =m_1$ and using Theorem \ref{mr} and Lemma \ref{j1}, we have
 \begin{align*}&\dim \mathcal{M}(L)+m+1+1\leq \dim \mathcal{M}(L)+\dim L^2+\dim \text{Im} \gamma_2' +\dim \text{Im} \gamma_3' \\&\leq \dim (L/L^2\wedge L/L^2)+ \dim (L^2/L^{3}\otimes_{mod} L^{(ab)})+\dim (L^3\otimes_{mod} L^{(ab)})= \\&\dfrac{1}{2}6+3(m-m_1)+3m_1.\end{align*} Therefore $\dim \mathcal{M}(L)\leq 2m+1.$
On the other hand,  since $n=m+3$, by the assumption $\dim \mathcal{M}(L)=2m+2$, which is a contradiction.
 This completes  the proof.
\end{proof}
\begin{thm}\label{mai}
There is no $n$-dimensional nilpotent Lie algebra of nilpotency class $c\geq 3$  and $ \dim L^2=m $ when $\dim  \mathcal{M}(L)=\dfrac{1}{2}(n+m-2)(n-m-1)+1.$ In particular, $\dim \mathcal{M}(L)\leq \dfrac{1}{2}(n+m-2)(n-m-1),$ for all Lie algebras of nilpotency class $c\geq 3$.
\end{thm}
\begin{proof}
  Let there be such a Lie algebra $L$. We get a contradiction by using induction on $c.$  By using Proposition \ref{hem}, there is no $n$-dimensional nilpotent Lie algebra $L$ of nilpotency class $3$ such that $\dim \mathcal{M}(L)=\dfrac{1}{2}(n+m-2)(n-m-1)+1.$ Now let $c>3.$ By using the induction hypothesis, $L/Z(L)$ cannot obtain the upper bound given in Theorem \ref{5}, which is impossible by looking Proposition \ref{fjj19}. Therefore the assumption is false and the result obtained.
\end{proof}

In the following examples, we may use the method of Hardy and Stitzinger in \cite{ha} to show that there are some Lie algebras of dimension $n$ and $\dim L^2=m$ such that $\dim \mathcal{M}(L)= \dfrac{1}{2}(n+m-2)(n-m-1)$.
\begin{ex}
 \[\text{Let}~L_{5,7}=\langle x_1,\ldots, x_5 | [x_1, x_2] = x_3, [x_1, x_3] = x_4,[x_1, x_4] = x_5\rangle,~\text{and}\]\[L_{5,9}=\langle   x_1,\ldots, x_5 | [x_1, x_2] = x_3, [x_1, x_3] = x_4, [x_2, x_3] = x_5 \rangle.\]
 Computing the Schur multiplier as in Hardy and Stitzinger in \cite{ha} yields $\dim \mathcal{M}(L_{5,7})=\dim \mathcal{M}(L_{5,9})=3$.
 Therefore $L_{5,7}$ and $L_{5,9}$ obtain the upper bound mentioned in Theorem \ref{mai}.

\end{ex}

\end{document}